\documentclass[12pt]{amsart}

\usepackage{amsthm,amsmath,amstext,amsfonts,amssymb}
\usepackage{fullpage}
\usepackage{url}
\usepackage{txfonts} 
\usepackage{mathrsfs}
\usepackage[osf]{mathpazo}

\DeclareMathOperator{\Vol}{\operatorname{Vol}}
\newcommand{\Z}{\mathbb{Z}}
\newcommand{\R}{\mathbb{R}}

\DeclareMathOperator{\Tr}{\operatorname{Tr}}
\newcommand{\GL}{\mathbf G\mathbf L}
\newcommand{\SL}{\mathbf S\mathbf L}

\DeclareMathOperator{\grad}{grad}
\DeclareMathOperator{\hess}{hess}
\DeclareMathOperator{\id}{id}

\def\abs#1{\lvert#1\rvert}
\def\norm#1{\lVert#1\rVert}
\def\set#1{\left\{#1\right\}}
\def\br#1{\left[#1\right]}
\def\pr#1{\left(#1\right)}

\newcommand{\eg}{{\it e.g. }}
\newcommand{\ie}{{\it i.e. }}

\newcommand{\via}{{\it via }}

\theoremstyle{plain}

\newtheorem{theorem}{Theorem}[section]

\newtheorem{corollary}[theorem]{Corollary}
\newtheorem{lemma}[theorem]{Lemma}

\theoremstyle{definition}
\newtheorem{defn}[theorem]{Definition}

\theoremstyle{definition}
\newtheorem*{defn*}{Definition}

\theoremstyle{remark}

\theoremstyle{remark}

\theoremstyle{remark}

\theoremstyle{remark}
\newtheorem*{ex*}{Example}
\theoremstyle{remark}
\newtheorem*{exs*}{Examples}

\begin{document}

\title[Energy, periodic sets and spherical designs]{Energy minimization, periodic sets and spherical designs}

\author[R.~Coulangeon and A.~Sch\"urmann]{Renaud Coulangeon and Achill Sch\"urmann}




\address{
Universit\'e de Bordeaux,
Institut de Math\'ematiques,
351, cours de la
\linebreak\indent
Lib\'e\-ration,
33405 Talence cedex, France}
\email{Renaud.Coulangeon@math.u-bordeaux1.fr}

\address{
Institute of Mathematics, 
University of Rostock,
18051 Rostock,
Germany}
\email{achill.schuermann@uni-rostock.de}

\begin{abstract}
We study energy minimization for pair potentials among periodic sets in Euclidean spaces. 
We derive some sufficient conditions under which a point lattice locally minimizes
 the energy associated to a large class of potential functions. 
This allows in particular to prove a local version of Cohn and Kumar's conjecture
that $\mathsf{A}_2$,  $\mathsf{D}_4$, $\mathsf{E}_8$ and the Leech lattice are globally universally optimal,
regarding energy minimization, and among periodic sets of fixed point density.
\end{abstract}

\maketitle

\tableofcontents

\section{Introduction}

The study of point configurations that minimize energy given by some pair potential 
occurs in diverse contexts, such as crystallography, electrostatics or computer graphics. 
There exist numerous numerical approaches to find locally optimal or stable configurations.
However, a mathematical rigorous treatment proving optimality of a point configuration
is quite difficult. 
 
Already in 1897, J.J. Thomson, the inventor (discoverer) of the electron,
came to the conclusion that ``the equations which determine the stability
of such a collection of particles increase so rapidly in
complexity with the number of particles that a general
mathematical investigation is scarcely possible''. 
In some special situations more can be said though.

\subsection{Energy minimizing spherical codes}

One important case that has been studied extensively by physicists are
point configurations (charged particles) on the surface of a sphere. We call such sets \textit{spherical codes} in what follows. We may consider the unit sphere $\mathcal S^{d-1} \subset \R^d$.
Given a  real-valued nonnegative function~$f : (0,4] \to \R$,
we ask in this situation to minimize the $f$-potential energy
\begin{equation}\label{nrg} 
E(f,\mathcal C)=\sum_{x \neq y \in \mathcal C} f(\norm{x-y}^2)
\end{equation}
among point sets $\mathcal C$ on $\mathcal S^{d-1}$ 
of fixed cardinality $|C|$.

One may think of $\mathcal C$ as a set of electrically charged particles in $\R^3$, 
and of the potential $f(r) = r^{1/2}$, in which case definition (\ref{nrg}) coincides with the classical notion of potential energy in physics.
Depending on the function $f$, the solutions may look quite different.
However, for many ``reasonable functions'' $f$, we may get the same solution.
For example, the vertices of a regular simplex (tetrahedron in $\R^3$) 
will be optimal for any continuous and decreasing function~$f$.

In \cite{MR2257398}, Cohn and Kumar introduced the notion of a {\em universally optimal configuration} of points. 
On the unit sphere, they are point configurations~$\mathcal C$ that minimize $E(f,\mathcal C)$ for all \textit{completely monotonic functions},
that is, for all real-valued $\mathcal C^{\infty}$ functions on the interval $I=(0,4]$, 
such that $(-1)^k f^{(k)}(x) \geq 0$ for all $x$ in $I$ and all $k \geq 0$. The class of completely monotonic functions
contains in a sense all the ``reasonable functions'' in the context of energy minimization.
It in particular contains all the inverse power laws $f(r) = r^{-s}$ with $s>0$, which are often studied in physics.

It turns out that there are several fascinating examples of universally optimal spherical codes.
Indeed, Cohn and Kumar were able to derive a  sufficient criterion for universal optimality, via so-called linear programming bounds and they showed that
all spherical configurations with at most~$m$ mutual distances, that form a
spherical~$(2m-1)$-design (see Definition~\ref{design}), are universally optimal.
This shows by example that there are exceptional structures (and infinite families of them) for which a 
\textit{rigorous} mathematical proof of a very broad energy minimization property is possible.

\subsection{Energy minimizing periodic point sets}

Cohn's and Kumar's considerations for spherical point sets can
be extended in several ways. Note first  that definition (\ref{nrg}) makes sense for any finite set in $\R^d$. One further natural extension is to consider infinite sets $\mathcal C$ in Euclidean spaces, possibly unbounded. This however yields difficulties in defining potential energy properly, 
because of possible subtle convergence problems. For \textit{periodic sets} such problems can be avoided. 
We say a discrete set in $\R^d$ is periodic, if it is a disjoint union of finitely many translates of a given full-rank lattice $L \subset \R^d$ (a discrete subgroup of $\R^d$).
In particular, a full-rank-lattice $L \subset \R^d$ itself is a periodic set. 
In general we can write,
$$\Lambda=\bigsqcup_{i=1}^m (t_i+L)$$ 
where $t_1, \dots, t_m$ are some vectors in $\R^d$.
For a potential function $f$, the $f$-energy of~$\Lambda$ is defined as
\begin{equation}\label{pnrg} 
E(f,\Lambda)=\frac{1}{m}\sum_{i=1}^m\sum_{ x \in \Lambda, x \neq t_i} f(\norm{t_i-x}^2).
\end{equation}
It can be shown that, when finite, the right-hand side of (\ref{pnrg}) is equal to 
\begin{equation}\label{pnrgbis} 
\lim_{R \rightarrow +\infty }\frac{1}{\abs{\Lambda_R}}\sum_{ x , y\in \Lambda_R, x \neq y} f(\norm{x-y}^2)
\end{equation}
where $\Lambda_R=\set{ x \in  \Lambda , \norm{x} \leq R}$ (\cite[Lemma 9.1.]{MR2257398}). 
This clarifies the link with the definition of energy for finite configurations of points. 

As in the case of spherical codes, one may ask if there exist \textit{universally optimal periodic sets}, that is, periodic sets that minimize the energy $E(f,\Lambda)$ for all completely monotonic functions $f$. At this point, no such universally optimal periodic set is known. However, exceptional structures as the hexagonal lattice, the root lattice~$\mathsf{E}_8$,
and the $24$-dimensional Leech lattice are conjectured to be examples (see~\cite{MR2257398}).
Recent experiments show (see \cite{cks-2009})
that also the root lattice~$\mathsf{D}_4$ and (somewhat surprisingly) 
the periodic non-lattice set~$\mathsf{D}_9^+$ could be universally optimal.

As a first attempt to prove universal optimality for any of the examples above, it is natural to ask whether universal optimality holds at least \textit{locally}. Before we go further, we recall a few known results for a similar question in the noticeably simpler context of lattices. Indeed, when $\Lambda=L$ is a lattice, and $f(r)=\dfrac{1}{r^s}$ for some $s > \frac{n}{2}$ , then the corresponding energy $E(\dfrac{1}{r^s},L)$ coincides with the Epstein zeta function of $L$ 
$$
E(\dfrac{1}{r^s},L)=\zeta(L,s)=\sum_{0\neq x \in L}\norm{x}^{-2s}
.
$$

Similarly, if $f(r)=e^{-c r}$, the corresponding energy is 
$$
E(e^{-c r},L)=\sum_{0\neq x \in L} e^{-c\norm{x}^2 }=\theta_L(i\frac{c}{\pi}) - 1
,
$$
where $\theta_L$ is the usual theta series of $L$.
Questions of optimality for lattices with respect to their zeta (resp. theta) function have recently been investigated by Sarnak and Str\"ombergsson
in~\cite{MR2221138}, and by the first author in \cite{MR2272103}, in connection with the theory of spherical designs. In particular, one has the following sufficient condition for local optimality among lattices:
\begin{theorem}[\cite{MR2272103}]\label{coul} \hspace*{1cm}\\
Lattices for which all shells are $4$-designs achieve a local mimimum 
(among lattices) of the map $L \mapsto E(e^{-c r},L)$ for big enough~$c$.
\end{theorem}
Cohn and Kumar observed in \cite{MR2257398} 
that it is enough to deal with potentials of the type $r \mapsto  e^{-c r}$, $c>0$, 
to recover all completely monotonic potentials. 
So, in view of the above theorem, universal optimality among lattices essentially reduces to a property of theta series.
Nevertheless, to actually infer universal local optimality (among lattices) from Theorem \ref{coul}, one has to be able in addition to remove the restriction to ``big enough~$c$'', and get the result ``for any $c>0$'' instead, which is highly non trivial in general. It turns out to be possible however in the case of $\mathsf{A}_2$,  $\mathsf{D}_4$, $\mathsf{E}_8$ and the Leech lattice, thanks to Sarnak and Str\"ombergsson's result in~\cite{MR2221138} (see~\cite[Theorem 1]{MR2221138}).
 
\bigskip

As in the case of spherical point sets, where the kissing number problem can be seen as a limiting case,
the sphere packing problem (asking for the maximum possible minimum distance of points 
at a fixed point density) is a limiting case of energy minimization of Euclidean point sets. It can be shown that the density of \textit{periodic} packings come arbitrarily close to the optimal density of a sphere packing in a given dimension $d$. Whereas the local optima for the density of \textit{lattice} packings  are well understood through Voronoi's characterization in terms of perfection and eutaxy, the situation for \textit{periodic} packings is comparatively more difficult. It was shown in \cite{schuermann-2010} that if a lattice $L$ is \textit{perfect} and \textit{strongly eutactic} (\ie the minimal vectors form a $2$-design), then $L$ achieves a local maximum for density not only among \textit{lattice} sphere packings but also among all \textit{periodic} sphere packings \cite[Theorem 10]{schuermann-2010}. By a theorem of Venkov, the condition that $L$ be perfect and strongly eutactic is satisfied in particular when the set of minimal vectors forms a $4$-design. Lattices satisfying this property are sometimes called \textit{strongly perfect} in the literature. To summarize, one has
\begin{theorem}[\cite{schuermann-2010}]\hspace*{1cm}\\
Lattices for which the set of minimal vectors forms a $4$-design
achieve a local optimum for the sphere packing problem
among \emph{all} periodic sets.
\end{theorem}
In this respect, strongly perfect lattices are somehow extremely rigid : there is no possibility to improve locally their density within the set of periodic sets.

The aim of this paper is to combine the ideas of~\cite{MR2272103} and~\cite{schuermann-2010} to prove essentially   that lattices satisfying the conditions of Theorem\ref{coul} are locally universally optimal not only among lattices, but indeed among all periodic sets (a precise formulation is given in Section \ref{mr}, Theorem~\ref{main}). Again, this means that the $4$-design property yields a strong rigidity. Our main result (Theorem~\ref{main}) can be in particular applied to the lattices~$\mathsf{A}_2$,~$\mathsf{D}_4$, $\mathsf{E}_8$
and to the Leech lattice (see Theorem~\ref{thm:application}) : this generalizes the result of Sarnak and Strombergsson in~\cite{MR2221138}, and proves a local version of Cohn and Kumar's conjecture~\cite[Conjecture 9.4]{MR2257398}.

\section{Preliminaries.}
 \subsection{A space of parameters.}
The study of local variations of energy first requires a suitable parametrization of the space of periodic sets. From now on, unless otherwise stated, the word "lattice" will stand  for "full-rank Euclidean lattice", \ie for a discrete subgroup of maximal rank in the Euclidean space $\R^d$, equipped with its standard norm $\| u \|=\left(\sum_{i=1}^d u_i^2\right)^{\frac{1}{2}}$. Following \cite{schuermann-2009} we say that $\Lambda\subset \R^d$ is an $m$-periodic set if there exists a lattice $L\subset \R^d$ and vectors $t_1, \dots, t_m$ in $\R^d$ such that 
\begin{equation}\label{per} 
\Lambda=\bigsqcup_{i=1}^m (t_i+L),
\end{equation}
the disjoint union of $m$ translates of $L$ (in other words, we assume that the sets $t_i+L$ are pairwise disjoint, \ie $t_i-t_j \notin L$ for $i\neq j$). We denote by $\mathcal L_m$ the set of $m$-periodic sets in $\R^d$.

We define the \textit{point density} of $\Lambda=\bigsqcup_{i=1}^m t_i+L$ as
\begin{equation}
 p\delta(\Lambda)=\dfrac{m}{\sqrt{\det L}}.
\end{equation} 
This accounts for the number of points per unit volume and is of course independent of the representation of $\Lambda$ as a union of translates of a lattice (we use the terminology \textit{point density} rather than simply \textit{density} to avoid any confusion with the density of the associated sphere packings).

Since most of the quantities we will be considering (\eg energy, packing-density) are invariant under orthogonal transformations and translations, we may identify two $m$-periodic sets which are isometric. In particular,  the $m$-tuple $\left( t_1, \dots, t_{m}\right)$ can be defined up to translation of its components by a common vector.  In what follows, we adopt the notation $\R^{md}_{*}$ to refer to the set of $m$-tuples $\mathbf{u}=\left( u_1, \dots, u_{m}\right)$ of vectors in $\R^d$ subject to the condition 
\begin{equation}\label{diag} 
u_i-u_j \notin \Z^d \text{ for } i \neq j
\end{equation} 
and we denote by $\R^{md}_{*}/\mathsf{T}$ the same set up to translation. For any $\mathbf{u}=\left( u_1, \dots, u_{m}\right) \in \R^{md}_{*}$, we define a \emph{standard periodic set}
\begin{equation}
\Omega_{\mathbf{u}}=\bigsqcup_{i=1}^m (u_i+\Z^d).
\end{equation}
Then, any $m$-periodic set may be written as $A\Omega_{\mathbf{u}}$ for some $A \in \GL_d(\R)$ and $\mathbf{u} \in \R^{md}_{*}$. 
The matrix $A$ in the above expression is determined, up to left multiplication by $O(d)$, 
by the positive definite quadratic form $Q=A^tA$. Note that we use column vectors, and with these settings have
$$\|A x \|^2 = Q [x] \coloneqq x^t Q x .$$
Using the notation $\mathcal S ^d$ for the set of $d\times d$ real symmetric matrices and $\mathcal S ^d_{>0}$ for the cone of positive definite ones, we thus get a parametrization of $O(d)\backslash\mathcal L_m\slash\mathsf{T}$ by $\mathcal S ^d_{>0} \times \R^{md}_{*}\slash\mathsf{T}$:
to $\left( Q,\mathbf{u} \right) \in \mathcal S ^d_{>0}\times \R^{md}_{*}/\mathsf{T}$ one associates
the $m$-periodic set $\Lambda=A\Omega_{\mathbf{u}}$, where $A$ is a square root of $Q$.
In keeping with \cite{schuermann-2009}, the elements of $\mathcal S ^{d,m}_{>0}\coloneqq \mathcal S ^d_{>0}\times \R^{md}_{*}$ are called $m$-\textit{periodic forms}.

Finally, energy comparison between two $m$-periodic sets makes sense only if they are assumed to have the same point density (otherwise, by shrinking/expanding  a given periodic set with a scaling factor, one can achieve arbitrarily small/large energy). One can for instance restrict to $m$-periodic sets of point density $m$, which amounts, in the above parametrization by periodic forms, to consider the space $\mathcal P ^{d,m}_{>0}\coloneqq \mathcal P ^d_{>0}\times \R^{md}_{*}/\mathsf{T}$, where $\mathcal P ^d_{>0}$ stands for the set of positive definite quadratic forms of determinant $1$. 
 
In accordance with formula (\ref{pnrg}) or (\ref{pnrgbis}), computation of energy involves evaluating potential functions over the set of nonzero elements in $$\Lambda-\Lambda \coloneqq  \set{ x-y  \;  : \;  x, y \in \Lambda}.$$
One difficulty is that a given element in $\Lambda-\Lambda$ generally admits several representations as a difference of two elements in $\Lambda$. The situation is somewhat simpler when $\Lambda$ is a lattice, as shown by the following lemma.
 
\begin{lemma}\label{l1} 
Let $\Lambda=\displaystyle \bigsqcup_{i=1}^m t_i+L$ be an $m$-periodic set in $\R^d$. For $x \in \Lambda$, set $$\Lambda_x=\set{y-x  \; : \; y \in \Lambda}.$$ 
Then the following assertions are equivalent
\begin{enumerate}
\item \label{l11} $\Lambda$ is a lattice.
\item \label{l12} $\Lambda-\Lambda=\Lambda$.
\item \label{l13} $\Lambda_x=\Lambda$ for all $x\in \Lambda$. 
\item \label{l14} For any $k$ in $\set{1,\dots,m}$, there is a uniquely defined permutation $\sigma_k$  of $\set{1,\dots,m}$ such that  
$$\forall i \in \set{1,\dots,m} \quad t_{\sigma_k(i)}\equiv t_i-t_k \mod L.$$ 
\end{enumerate}
\end{lemma}
\begin{proof} 
Lattices are characterized as discrete additive subgroups of $\R^d$
and the equivalence of (\ref{l11}), (\ref{l12})  and (\ref{l13}) is derived from that. As for (\ref{l13}) $\Rightarrow$ (\ref{l14}), we have that for fixed $k$, the difference $t_i-t_k$ lies in $\Lambda_{t_k}=\Lambda=\bigsqcup_{j=1}^m t_j+L$, so there exists a uniquely determined index $\sigma_k(i)$ such that $t_i-t_k \in t_{\sigma_k(i)}+L$. Moreover, $\sigma_k(i) =\sigma_k(j)$ if and only if $t_i -t_k \equiv t_j -t_k \mod L$, which means that $t_i-t_j \in L$, whence $i=j$, so $\sigma_k$ is a bijection. Finally, property (\ref{l14}) cleary implies that any pairwise differences of elements in $\Lambda$ are in $\Lambda$, which shows that (\ref{l14}) $\Rightarrow$ (\ref{l11}).
\end{proof}

\subsection{Potentials.}\label{pot} 
As regards the potential functions $f$ to be used, the following assumptions will be made throughout:\medskip

\textbf{Assumption 1.} \textit{$f$ is a completely monotonic function on $(0,\infty)$, \ie real-valued $\mathcal C^{\infty}$ functions on $(0,\infty)$ such that $(-1)^k f^{(k)}(x) \geq 0$ for all $x$ in $(0,\infty)$ and all $k \geq 0$ (in particular, $f$ is nonnegative).}

Such functions will be referred to as CM functions in what follows.\medskip

It will also be useful, although not necessary, to assume that formula (\ref{pnrg}) converges, which is ensured by the following\medskip

\textbf{Assumption 2.} \textit{There exists $\epsilon >0$ such that $f(x) =O(x)^{-\frac{d}{2}-\epsilon}$ as $x$ tends to infinity.}
\medskip

It would be possible, using Bernstein's theorem, to restrict to potentials of the form $f_c(r)=e^{-cr}$ with $c>0$. Indeed, any CM function may be written as
$$f(x)=\int_{0}^{\infty}e^{-cx} d\alpha(c)$$ (Stieljes Integral) for some weakly increasing function $\alpha$ (see \cite{MR0005923}[Theorem 12b, p. 161]).

A case that we consider separately on its own first, is that of inverse power laws $p_s(r)=r^{-s}$ for some $s>0$.
These do not encompass the whole class of CM functions, but they are easier to deal with.

\section{Local Study of Potential Energy.}
\subsection{Local expression for the energy}

The $f$-energy of an $m$-periodic set $\Lambda=A\Omega_{\mathbf{u}}$ depends only on the associated periodic form $(Q,\mathbf{u})$, namely one has
\begin{equation}\label{pnrgter} 
E(f,\Lambda)=E(f,(Q,\mathbf{u}))=\frac{1}{m}\sum_{i=1}^m\sum_{ x \in \Omega_{\mathbf{u}}, x \neq u_i} f(Q\br{u_i-x}).
\end{equation}
We want to expand the $f$-energy in a neighbourhood of a given $m$-periodic set $$\Lambda_0=A_0\Omega_{\mathbf{u}^0}=\sqcup_{i=1}^m t^0_i+L_0,$$ where we set $L_0=A_0 \Z^d$ and  $\mathbf{t}^0= A_0\mathbf{u}^0$ (\ie $t^0_i=Au^0_i$ for $1 \leq i \leq m$). We also assume that $\Lambda_0$ has point density $p\delta(\Lambda_0)=m$ and we let $X_0=(Q_0,\mathbf{u}^0)$ be the corresponding periodic form, with $Q_0=A_0^tA_0$.

The manifold $\mathcal P ^{d,m}_{>0}= \mathcal P ^d_{>0}\times \R^{md}_{*}/\mathsf{T}$ is locally homeomorphic in a neighborhood of $X_0=(Q_0,\mathbf{u}^0)$ to its tangent space  at $X_0$ which is identified with $\mathcal T_{Q_0}\times \R^{md}/\mathsf{T}$ where
$$\mathcal T_{Q_0}= \set{K\in \mathcal S^d \; : \; \Tr(Q_0^{-1}K)=0}.$$ 
The isomorphism is obtained \via the matrix exponential through the map $\left( K ,\mathbf{u}\right) \mapsto \left( Q_0 \exp(Q_0^{-1}K),\mathbf{u}^0+\mathbf{u}\right)$. Note that the tangent space $\mathcal T_{Q_0}\times \R^{md}/\mathsf{T}$ at  $X_0$ comes equipped with its standard $\SL_d(\R)$-invariant scalar product
\begin{equation}\label{sp} 
 \left\langle \left( K ,\mathbf{u}\right)  ,\left( L ,\mathbf{v}\right) \right\rangle_{X_0} \coloneqq \Tr(Q_0^{-1}KQ_0^{-1}L) + \sum_{i=1}^m u_i^tv_i
\end{equation} 
which defines the Riemannian structure of $\mathcal P ^{d,m}_{>0}$.
To study the local variations of the $f$-energy around $X_0$, it is enough to consider the $f$-energy of 
$(Q_0 \exp(Q_0^{-1}H),\mathbf{u}^0+\mathbf{u})$, for small enough $H \in \mathcal T_{Q_0}$ and $\mathbf{u} \in \R^{md}/\mathsf{T}$. It equals
\begin{equation}\label{nrgexp} 
\frac{1}{m}\sum_{i=1}^m\sum_{ \substack{x \in \Omega_{\mathbf{u}^0 + \mathbf{u}}\\ x \neq u^0_i+u_i}} f(Q_0 \exp(Q_0^{-1}K)\br{u^0_i+u_i-x})
\end{equation}
Each term $u^0_i+u_i-x$ in the internal sum may be written as $u^0_i-u^0_j+u_i-u_j+v$ for some $j \in \set{1,\dots,m}$ and some $v \in \Z^d$. Note that the condition $u^0_i-u^0_j+u_i-u_j+v \neq 0$ will be satisfied if and only if $u^0_i-u^0_j+v$ itself is non-zero, provided that the $u_i$ are close enough to $0$ (this will be the case for instance if the $u_i$'s satisfy $\norm{u_i}<\frac{\rho_0}{2}$, where $\rho_0 := \min_{0 \neq x \in \Omega_{\mathbf{u}^0} - \Omega_{\mathbf{u}^0}}\norm{x}$). Consequently, assuming that $\mathbf{u}$ lies in a suitable neighbourhood of $0$, we can rewrite (\ref{nrgexp}) as
\begin{equation}\label{nrg2} 
 \frac{1}{m} \; \sum_{1\leq i,j\leq m} \; \sum_{0\neq w \in u^0_i-u^0_j+\Z^d}  f(Q_0 \exp(Q_0^{-1}K)\br{w+u_i-u_j})
\end{equation}

 In order to get simpler expressions in the calculations to come it is more convenient to change coordinates, that is we rewrite the above expression as
\begin{equation}\label{nrg3} 
E_f(H,\mathbf{t}) \coloneqq\frac{1}{m}  \; \sum_{1\leq i,j\leq m} \; \sum_{0\neq w \in t^0_i-t^0_j+L_0}  f(\exp(H)\br{w+t_i-t_j})
\end{equation}
where $\mathbf{t}=A_0 \mathbf{u}$ (resp. $\mathbf{t}^0=A_0 \mathbf{u}^0$), and $H=\left( A_0^{-1}\right)^t K A_0^{-1}$ is now in $\mathcal T_{\id}=\set{H \in \mathcal S^d \; : \; \Tr(H)=0}$. To see that (\ref{nrg2}) and (\ref{nrg3}) are the same, we use the identity 

$$Q_0 \exp(Q_0^{-1}K)[x]=x^t Q_0 \exp(Q_0^{-1}K) x = x^t A_0^t \exp\left( \left(A_0^{-1}\right)^t K  A_0^{-1}\right) A_0 x=\exp(H)[A_0 x]$$
valid for any $x \in \R^d$. That $H=\left( A_0^{-1}\right)^t K A_0^{-1}$ is symmetric is clear, and $\Tr H=0$ follows from the simple observation that $$\Tr\left( \left( A_0^{-1}\right)^t K A_0^{-1}\right) =\Tr\left(A_0^{-1} \left( A_0^{-1}\right)^t K \right) =\Tr(Q_0^{-1}K).$$
 Note that the scalar product (\ref{sp}) on $\mathcal T_{\id}$, which
 we denote simply by $\left\langle \, ,\right\rangle$ in what follows, takes the form 
\begin{equation}
 \left\langle \left( K ,\mathbf{u}\right)  ,\left( L ,\mathbf{v}\right) \right\rangle = \Tr(KL) + \sum_{i=1}^m u_i^tv_i .
\end{equation}
Note that the definition of~$E_f$ depends on a given representation of~$\Lambda_0$ as a periodic set, that is,
it depends on $A_0$ and~$\mathbf{u}^0$. 
Note also that $E(f,\Lambda) = E_f(0,\mathbf{0})$ with this setting.

The two main ingredients to obtain further simplifications in the above formula are the following :
\begin{enumerate}
 \item use the additive structure of $\Lambda_0$ (if any).
\item use translation invariance of the energy.
\end{enumerate}

These conditions are met in particular when $\Lambda_0$ is a lattice, in which case we obtain the following crucial lemma :

\begin{lemma}
Suppose that $\Lambda_0=\displaystyle \bigsqcup_{i=1}^m (t^0_i+L_0)$ is a lattice. Then 
\begin{equation}\label{latnrg} 
E_f(H,\mathbf{t})
=\frac{1}{m^2}\sum_{0\neq w \in \Lambda_0} \sum_{1\leq i,k \leq m}f\pr{\exp(H)\br{w+t_i-t_{\sigma_k(i)}}}
,
\end{equation}
where ${\sigma_k}$ is the bijection of Lemma~\ref{l1}. 
\end{lemma}
\begin{proof}
Since $\Lambda_0 - \Lambda_0=\Lambda_0$, any coset $t^0_i-t^0_j+L_0$ in the internal sum (\ref{nrg3}) can be written as $t^0_k + L_0$ for a uniquely defined~$k$. 
More precisely, using Lemma~\ref{l1}(\ref{l13}), we obtain

\begin{equation}\label{nrg4} 
E_f(H,\mathbf{t})
=\frac{1}{m}\sum_{k=1}^m \; \sum_{0\neq w \in t^0_k+L_0 }  \; \sum_{i=1}^m f\pr{\exp(H)\br{w+t_i-t_{\sigma_k(i)}}} 
\end{equation}
where $\sigma_k$ is the permutation defined by the condition that $t^0_{\sigma_k(i)}\equiv t^0_i - t^0_k \mod L_0$ for all $i \in \set{1,\dots,m}$ (see Lemma~\ref{l1}).
Note that the $t_j$ are replaced by $t_{\sigma_k(i)}$ and that the change from index $j$ to $k$ causes a reordering of terms.

Because of the translation invariance of the energy, the energy is not modified if all the components of $\mathbf{t}^0$ are translated by a common vector $\alpha \in \R^d$. In particular, we can choose $\alpha= -t^0_j$ for some $j \in \set{1,\dots, m}$. Applying this to (\ref{nrg4}), we get for any $j \in \set{1,\dots, m}$, the equation
 
\begin{align}
 \tag*{$\mathbf{(E_j)}$}\label{nrg5}\!\!\!
E_f(H,\mathbf{t})  
&=&\frac{1}{m}\sum_{k=1}^m \; \sum_{0\neq w \in t^0_k-t^0_j+L_0 } \; \sum_{i=1}^m f\pr{\exp(H)\br{w+t_i-t_{\sigma_k(i)}}}\\
\notag  &=&\frac{1}{m}\sum_{k=1}^m \; \sum_{0\neq w \in -t^0_{\sigma_k(j)}+L_0 } \; \sum_{i=1}^m f\pr{\exp(H)\br{w+t_i-t_{\sigma_k(i)}}}.
\end{align}
Adding up the \ref{nrg5}s for $j =1, \dots, m$ and then averaging, together with the observation that $$\bigsqcup_{j=1}^m-t^0_{\sigma_k(j)}+L_0 =\bigsqcup_{j=1}^m -t^0_j+L_0=\bigsqcup_{j=1}^m t^0_j+L_0=\Lambda_0,$$ 
we obtain the final expression 
\begin{align*} 
E_f(H,\mathbf{t})  
&=&\frac{1}{m^2}\sum_{k=1}^m \; \sum_{0\neq w \in \Lambda_0} \; \sum_{i=1}^m f\pr{\exp(H)\br{w+t_i-t_{\sigma_k(i)}}}\\
\notag &=&\frac{1}{m^2} \; \sum_{0\neq w \in \Lambda_0} \; \sum_{1\leq i,k \leq m}f\pr{\exp(H)\br{w+t_i-t_{\sigma_k(i)}}}
\end{align*} 
\end{proof}

\subsection{Taylor expansion of the energy}
We compute in this section the Taylor expansion of order $2$ of (\ref{latnrg}), viewed as a function on $\mathcal P ^{d,m}_{>0}$. To that end, we need to compute the gradient and Hessian of $E_f$ at a lattice $\Lambda_0$, respectively at $(0,\mathbf{0})$, and then use the approximation
 \begin{equation*}
E_f(H,\mathbf{t}) 
=
E_f(0,\mathbf{0}) + \langle \grad E_f(0,\mathbf{0}), \pr{H,\mathbf{t}} \rangle + \frac{1}{2} \hess E_f (0,\mathbf{0}) [H,\mathbf{t}]  \;\; + o (\|\pr{H,\mathbf{t}}\|^2).
\end{equation*} 
The relevant quantities  are given by the following lemma, when $f$ is either an exponential $f_c$ or an inverse power law~$p_s$, which will be the only cases of interest in the sequel.

\begin{lemma}\label{gradhess} Suppose that $\mathbf{t}^0$ is such that $\Lambda_0=\bigsqcup_{i=1}^m t^0_i+L_0$ is a lattice in $\R^d$. Then 
\begin{enumerate}
\item For an inverse power law $p_s(r)=r^{-s}$, one has\medskip
\begin{eqnarray*} 
\langle \grad E_{p_s}(0,\mathbf{0}), \pr{H,\mathbf{t}} \rangle &=&-s\sum_{0\neq w \in \Lambda_0} H\br{w}{\|w\|}^{-2s-2}\\ 
\hess E_{p_s}(0,\mathbf{0})\br{H,\mathbf{t}}&=&
s \sum_{0\neq w \in \Lambda_0} 
{\|w\|}^{-2s-4}\left\lbrace \frac{s+1}{2}(H\br{w})^2-\frac{1}{2}H^2\br{w}{\|w\|^2}\right.\\
&& \left. +\frac{1}{m^2} \sum_{1\leq i,k \leq m} 2(s+1)\pr{w^{t} (t_i-t_{\sigma_k(i)})}^2-{\|w\|^2} \| t_i-t_{\sigma_k(i)} \|^2 \right\rbrace 
\end{eqnarray*}
\item For an exponential law $f_c(r)=e^{-cr}$, one has\medskip
\begin{eqnarray*}
\langle \grad E_{f_c}(0,\mathbf{0}), \pr{H,\mathbf{t}} \rangle &=&-c \sum_{0\neq w \in \Lambda_0} H[w]  e^{-c \|w\|^2}\\  
\hess E_{f_c}(0,\mathbf{0})\br{H,\mathbf{t}}&=&c\sum_{0\neq w \in \Lambda_0} 
  e^{-c \|w\|^2} \left\lbrace \frac{c}{2}  (H[w])^2-\frac{1}{2}H^2\br{w} \right. \\
&&\left. + \frac{1}{m^2}\sum_{1\leq i,k \leq m}2c\left( w^{t}(t_i-t_{\sigma_k(i)})\right)^2  - \|t_i-t_{\sigma_k(i)}\|^2 \right\rbrace 
\end{eqnarray*}
\end{enumerate}
\end{lemma}

\begin{proof} 
Using the Taylor expansion of the matrix exponential we write
\begin{equation*}
\exp(H)[w+t_i-t_{\sigma_k(i)}]={\|w\|^2}+\mathcal{L}\pr{H,\mathbf{t}}+\mathcal{S}\pr{H,\mathbf{t}}+ o (\|\pr{H,\mathbf{t}}\|^2)
\end{equation*}
where
\begin{equation*}
\mathcal{L}\pr{H,\mathbf{t}}=H[w]+ 2w^{t}(t_i-t_{\sigma_k(i)})
\end{equation*}
and
\begin{equation*}
\mathcal{S}\pr{H,\mathbf{t}}=\|t_i-t_{\sigma_k(i)}\|^2+2w^{t}H(t_i-t_{\sigma_k(i)})+\frac{1}{2}H^2\br{w}.
\end{equation*}
Expanding gives
\begin{eqnarray*}
\exp(H)[w+t_i-t_{\sigma_k(i)}]^{-s}  
& = &
{\|w\|}^{-2s}\left(1 + \frac{\mathcal{L}}{{\|w\|^2}} + \frac{\mathcal{S}}{{\|w\|^2}} \right)^{-s}+ o (\|\pr{H,\mathbf{t}}\|^2)  \\
& = & 
{\|w\|}^{-2s}\left(1-s\frac{\mathcal{L}+\mathcal{S}}{{\|w\|^2}} +\frac{s(s+1)}{2}\frac{\mathcal{L}^2}{{\|w\|}^4}\right)+ o (\|\pr{H,\mathbf{t}}\|^2)  
\end{eqnarray*}
in the first case, and
\begin{equation*}
e^{-c\exp(H)[w+t_i-t_{\sigma_k(i)}]}=e^{-c\|w\|^2}\left(1-c\left(\mathcal{L} +\mathcal{S} \right)+\frac{c^2}{2}\mathcal{L}^2\right)+ o (\|\pr{H,\mathbf{t}}\|^2) 
\end{equation*}
in the second one. 
Then, for a fixed $w \in  \Lambda_0$, one has to add the terms
$$\exp(H)[w+t_i-t_{\sigma_k(i)}]^{-s}, 
\qquad \mbox{resp.} \qquad
e^{-c\exp(H)[w+t_i-t_{\sigma_k(i)}]},$$ 
corresponding to all pairs $(i,k)$. 
Because $\sigma_k$ is a permutation, the terms  $2w^{t}(t_i-t_{\sigma_k(i)})$ appearing in $\mathcal L$ add up to zero, as do the terms $2w^{t}H(t_i-t_{\sigma_k(i)})$ in $\mathcal S$, and the terms $2w^{t}(t_i-t_{\sigma_k(i)})H[w]$ that show up in the expansion of $\mathcal L^2$. Altogether, this leads to the formulas of the lemma.
\end{proof}

There are two noticeable features in the previous calculations, whenever our periodic set $\Lambda_0$ actually is a lattice : the gradient of the potential energy at $\Lambda_0$, which is \textit{a priori} a linear form in the variable $(H,\mathbf{t}) \in \mathcal T_{\id} \times \R^{md}$, has a trivial component in the translational direction, and its Hessian splits into the sum of a quadratic form in $H$ and a quadratic form in $\mathbf{t}$. In other words, when studying local perturbations of energy within the set of periodic sets around a lattice, one can consider separatly \textit{purely translational} moves (\ie with $H=0$) and \textit{purely lattice} moves (\ie with $\mathbf{t}=\mathbf{0})$. This observation plays a prominent role in the results of the next section.

\noindent \textit{Remark.} The previous lemma extends partly to more general potential functions. For instance, one can show, using exactly the same argument as in the proof above, that whenever $f$ is a smooth function such that the potential energy $E(f,\Lambda)$ is defined and is a differentiable function on the space of periodic configurations, one has 
$$\langle \grad E_{f}(0,\mathbf{0}), \pr{H,\mathbf{t}} \rangle =\sum_{0\neq w \in \Lambda_0} H[w]  f'(\norm{w}) .$$

\section{Main Result}\label{mr} 
Using the preliminary computations of the previous section, we can
improve the results of~\cite{MR2272103}. We show that under some
rather general conditions, a lattice which is locally optimal
\emph{among lattices} regarding energy minimization, is in fact
locally optimal \emph{among all periodic sets}. One difficulty in giving a precise meaning to "optimal" or "critical point" for the energy, is that a given periodic set admits infinitely many representations of type $\Lambda=\bigsqcup_{i=1}^m t_i+L$, for various $m$'s and $L$'s. To overcome this problem, we adopt the following definition.

\begin{defn} Let $f$ be a CM function.
\begin{enumerate} 
\item A periodic set $\Lambda_0$ is $f$-critical if it is a critical point of~$E(f,\Lambda)$ on $\mathcal P ^{d,m}_{>0}$ for every~$m$.
\item A periodic set $\Lambda_0$ is locally $f$-optimal if it locally minimizes~$E(f,\Lambda)$ on $\mathcal P ^{d,m}_{>0}$ for every~$m$.
\end{enumerate}
\end{defn}

With this terminology, a periodic set is locally universally optimal if it is locally $f$-optimal for any CM~function~$f$, or equivalently, due to Bernstein's theorem, for any exponential potential $f_c$, $c>0$.

We will allow in some instances (\eg Theorem \ref{main} below) the wording \emph{$f$-critical} for a non necessarily CM function $f$. The least we need is that $f$ is smooth and decays sufficiently rapidly so that the potential energy is defined and is a differentiable function on the space of periodic configurations. This is the case in particular if $f$ satisfies \textbf{Assumption 2} of Section~\ref{pot}.

Besides the preliminary computations of the previous sections, the main tool we will use is the notion of \textit{spherical design}.

\begin{defn}\label{design} 
A finite set~$D$ of points on the sphere~$S_r$ of radius~$r$ in~$\R^d$ is a $t$-design if 
\begin{equation}
\dfrac{1}{\Vol (S_r)}\int_{S_r}f(x) \mathrm dx=\dfrac{1}{\abs{D}}\sum_{x \in D}f(x)
\end{equation}
holds for any polynomial $f$ of degree up to $t$.
\end{defn}
The following lemma, the proof of which may be found in  \cite{MR1881618} or \cite{MR2272103}, gives an alternative formulation of the $t$-design property that will be used throughout this section. 
\begin{lemma} [Venkov {\cite[Th\'eor\`eme 3.2.]{MR1881618}}]\label{venkov} 
Let $D$ be a finite subset of the sphere~$S_r$ of radius~$r$ in~$\R^d$ and $t$ an even positive integer. Assume that $D$ is symmetric about $0$, \ie $D=-D$. Then the following properties are equivalent :
\begin{enumerate} 
\item $D$ is a $t$-design.
\item \label{cd} There exists a constant $c_t$, depending only on $r$, $t$ and the cardinalty of $D$, such that $$\forall y \in \R^n , \quad \sum_{x \in D} (x \cdot \alpha) ^t=c_t (y \cdot y) ^{\frac{t}{2}}.$$
\end{enumerate} 
\end{lemma}

Our main result may be stated as follows

\begin{theorem}\label{main}  \quad
\begin{enumerate}
\item\label{un}   Let $\Lambda_0$ be a lattice, all shells of which are $2$-designs. Then, viewed as a periodic set, $\Lambda_0$ is $f$-critical for any CM function $f$, 
or more generally for any smooth function $f$ such that the potential energy $E(f,\Lambda)$ is defined and is a differentiable function on the space of periodic configurations.

\item\label{deux}  Let $\Lambda_0$ be a lattice, all shells of which are $4$-designs. Then, viewed as a periodic set, 
\begin{enumerate}
\item\label{deuxa}  $\Lambda_0$ is locally $p_s$-optimal for any $s>\frac{d}{2}$.
\item\label{deuxb}  $\Lambda_0$ is locally $f_c$-optimal for any big enough $c>0$.
\end{enumerate}
\end{enumerate}

\end{theorem}
\begin{proof}
For any fixed positive integer $m$, we write $\Lambda_0$ as an $m$-periodic set, say $\Lambda_0=\bigsqcup_{i=1}^m t_i^0 + L_0$.
We consider~$E_f$ (as in \eqref{nrg3}) depending on the particular choice of $L_0$ and $\mathbf{t}^0$
to locally study the energy in a neighbourhood of $\Lambda_0$ in $\mathcal P ^{d,m}_{>0}$.
We in particular use the Taylor expansion of $E_f$ around $(0,\mathbf{0})$ obtained in the previous section.

(\ref{un}) One has to show that for any CM function $f$, the gradient of $E_f$ at $(0,\mathbf{0})$ 
is orthogonal to $\mathcal T_{\id} \times \R^{md}$. Thanks to Bernstein's theorem, it is enough to show it for exponential functions $f_c$. For any $\alpha >0$, we set
$$
\Lambda_0(\alpha)= \set{w \in \Lambda_0 \; : \;  {\|w\|^2}=\alpha}
.
$$ 
These shells of the lattice $\Lambda_0$ are assumed to be $2$-designs (if non-empty). Using Lemma \ref{venkov}(\ref{cd}), this is easily seen to be equivalent to the relation
\begin{equation}\label{eut} 
\sum_{w \in \Lambda_0(\alpha)} ww^{t}=\frac{\alpha \abs{\Lambda_0(\alpha)}}{d}\id
\end{equation}
for every positive real number $\alpha$. In other words, the constant $c_2$ in Lemma \ref{venkov}(\ref{cd}) is equal to $\frac{\alpha \abs{\Lambda_0(\alpha)}}{d}$. Observing that $H[w]= \Tr (H w w^{t})$, the expression for the gradient of the energy obtained in Lemma \ref{gradhess} can be reformulated as

\begin{eqnarray*}
 \grad E_{f_c}(0,\mathbf{0})&=&-c \sum_{0\neq w \in \Lambda_0} ww^{t}  e^{-c {\|w\|}^2}\\&=&-c \sum_{\alpha >0}\sum_{w \in \Lambda_0(\alpha)} ww^{t}  e^{-c\alpha}.
\end{eqnarray*} 
Consequently, using (\ref{eut}) we obtain
\begin{equation}
\langle \grad E_{f_c}(0,\mathbf{0}), \pr{H,\mathbf{t}} \rangle  =-c \sum_{\alpha >0}  \frac{\alpha e^{-c\alpha}\abs{\Lambda_0(\alpha)}}{d}\Tr(H) =0
\end{equation}
for any $H \in \mathcal T_{\id}=\set{H \in \mathcal S^d \; : \; \Tr(H)=0}$.

Thanks to the remark following Lemma \ref{gradhess}, the previous computation extends readily to any smooth function $f$ such that the potential energy $E(f,\Lambda)$ is defined and is a differentiable function on the space of periodic configurations, since we then have 
\begin{equation*}
 \grad E_{f}(0,\mathbf{0})=\sum_{\alpha >0}f'(\alpha) \sum_{w \in
   \Lambda_0(\alpha)} ww^{t} 
.
\end{equation*}
Thus again $\langle \grad E_{f}(0,\mathbf{0}), \pr{H,\mathbf{t}} \rangle  = \sum_{\alpha >0}  \frac{f'(\alpha)\abs{\Lambda_0(\alpha)}}{d} \Tr(H) =0$ for any $H \in \mathcal T_{\id}$.

(\ref{deux}) To establish local optimality with respect to a given CM function $f$, it is enough to prove that $\hess E_f (0,\mathbf{0})$ is positive definite.  
By \cite[Proposition 1.2]{MR2272103}, the hypothesis that all shells of $\Lambda_0$ are $4$-designs translates into
\begin{equation}\label{4des} 
\forall H \in S_d (\R), \quad \sum_{0\neq w \in \Lambda_0(\alpha)} H\br{w}^2 =\frac{\alpha^2 \abs{\Lambda_0(\alpha)}}{d(d+2)}((\Tr H)^2 + 2 \Tr (H^2)),
\end{equation}
provided that $\Lambda_0(\alpha)$ is non-empty.
By the definition of a spherical design, it is clear that a $t$-design is automatically a $t'$-design if $t >t'$. Hence, all non-empty shells of $\Lambda_0$ are also 
$2$-designs, which implies in particular that
\begin{equation}\label{2des} 
\forall H \in S_d (\R), \quad \sum_{0\neq w \in \Lambda_0(\alpha)} H^2\br{w}=\frac{\alpha \abs{\Lambda_0(\alpha)}}{d}\Tr (H^2).
\end{equation}

In case $f=p_s$ is an inverse power function, we can plug (\ref{4des}) and (\ref{2des}) into the expression for $\hess E_{p_s} (0,\mathbf{0})$ obtained in Lemma \ref{gradhess}, which yields 
\begin{equation*}
\hess E_{p_s}(0,\mathbf{0})\br{H,\mathbf{t}}= \frac{s(s-\frac{d}{2})}{d(d+2)}\zeta(L_0,s) (\Tr H^2) + \frac{s}{m^2}\Psi_s(\mathbf{t}) 
\end{equation*}
where 
\begin{equation*}
\Psi_s (\mathbf{t}) =\sum_{0\neq w \in \Lambda_0} \left\lbrace \sum_{1\leq i,k \leq m}2(s+1)\pr{w^{t}(t_i-t_{\sigma_k(i)})}^2-{\|w\|^2} \|{t_i-t_{\sigma_k(i)}}\|^2 \right\rbrace {\|w\|}^{-2s-4}.
\end{equation*} 
Unless $H$ is zero, the first term $\frac{s(s-\frac{d}{2})}{d(d+2)}\zeta(L_0,s) \Tr (H^2)$ is positive 
because of the assumption $s > \frac{d}{2}$.
As for $\Psi_s (\mathbf{t})$, we can rewrite it as  
\begin{equation*}
\Psi_s (\mathbf{t}) =\sum_{\alpha >0}\sum_{0\neq w \in \Lambda_0(\alpha)} \left\lbrace \sum_{1\leq i,k \leq m}2(s+1)\pr{w^{t}(t_i-t_{\sigma_k(i)})}^2-\alpha \|{t_i-t_{\sigma_k(i)}}\|^2 \right\rbrace \alpha^{-s-2}
.
\end{equation*}
Since each non-empty shell of $\Lambda_0$ is a $2$-design, this simplifies to
\begin{equation*}
\Psi_s (\mathbf{t}) =\sum_{\alpha >0}\left( \frac{2(s+1)}{d}-1\right) \abs{\Lambda_0(\alpha)} \alpha^{-s-1}  \sum_{1\leq i,k \leq m}\|{t_i-t_{\sigma_k(i)}}\|^2 
\end{equation*}
which is obviously positive for $s > \frac{d}{2}$, unless $t_i=t_{\sigma_k(i)}$ for every $1\leq i,k \leq m$. Given that for any pair $(i,j)$ with $1\leq i,j \leq m$ there exists $k$ such that $\sigma_k(i)=j$ (namely $k=\sigma_j(i)$), this last condition implies that $t_i=t_j$ for all $(i,j)$ and consequently $\mathbf{t}\equiv \mathbf{0} \mod \mathsf{T}$. This proves assertion (\ref{deuxa}). 

If $f=f_c$ is an exponential potential, then the same kind of computation as before yields 
\begin{equation*}
\hess E_{f_c}(0,\mathbf{0})\br{H,\mathbf{t}}= \frac{\Tr (H^2)}{d(d+2)}\sum_{0\neq w \in \Lambda_0} c{\|w\|^2}\left(c{\|w\|^2}-\left( \frac{d}{2}+1\right) \right) e^{-c{\|w\|^2}}+ \frac{c}{m^2}\varUpsilon_c(\mathbf{t}) 
\end{equation*}
where 
\begin{equation*}
\varUpsilon_c(\mathbf{t}) =\sum_{0\neq w \in \Lambda_0} \left\lbrace \sum_{1\leq i,k \leq m}2c\pr{w^{t}(t_i-t_{\sigma_k(i)})}^2-\|{t_i-t_{\sigma_k(i)}}\|^2 \right\rbrace  e^{-c{\|w\|^2}}.
\end{equation*}

If $H\neq 0$, the first term $\frac{\Tr (H^2)}{d(d+2)}\sum_{0\neq w \in \Lambda_0} c{\|w\|^2}\left(c{\|w\|^2}-\left( \frac{d}{2}+1\right) \right) e^{-c{\|w\|^2}}$ is positive as soon as $c$ is strictly greater than $\frac{d+2}{2\min \Lambda_0}$, where we set 
$$
\min \Lambda_0=\min_{0\neq w \in \Lambda_0}{\|w\|^2}
.$$ 
On the other hand, due to the assumption that all non-empty shells of $\Lambda_0$ are $2$-designs, the expression of $\varUpsilon_c(\mathbf{t})$ reduces to
\begin{equation*}
\varUpsilon_c(\mathbf{t}) =\sum_{\alpha >0}\left( \frac{2c\alpha}{d}-1\right) \abs{\Lambda_0(\alpha)} e^{-c\alpha} \sum_{1\leq i,k \leq m}\|{t_i-t_{\sigma_k(i)}}\|^2.
\end{equation*}

This quantity is nonnegative for any $c > \dfrac{d}{2\min \Lambda_0}$, since it is then a sum of nonnegative terms, and it is zero if and only if $t_i=t_{\sigma_k(i)}$ for every $1\leq i,k \leq m$, that is $\mathbf{t}\equiv \mathbf{0} \mod \mathsf{T}$. The conclusion follows.
\end{proof}
 
 As a by-product of the previous proof we obtain that the Hessian of the $f_c$-potential energy splits into a sum
\begin{eqnarray*}
\hess E_{f_c}(0,\mathbf{0})\br{H,\mathbf{t}}&=&\frac{\Tr (H^2)}{d(d+2)}\sum_{0\neq w \in \Lambda_0} c{\|w\|^2}\left(c{\|w\|^2}-\left( \frac{d}{2}+1\right) \right) e^{-c{\|w\|^2}}\\&+&\frac{c}{m^2}\sum_{\alpha >0}\left( \frac{2c\alpha}{d}-1\right) \abs{\Lambda_0(\alpha)} e^{-c\alpha} \sum_{1\leq i,k \leq m}\|{t_i-t_{\sigma_k(i)}}\|^2,
\end{eqnarray*}  
whenever the $4$-design condition is satisfied on each shell of $\Lambda_0$.
Here the first term pertains to \textit{purely lattice} changes, and the second term to \textit{purely translational} ones. Setting $y=\dfrac{c}{\pi}$, we can rewrite it as
\begin{equation}
\hess E_{f_c}(0,\mathbf{0})\br{H,\mathbf{t}}=y\left[ \frac{\Tr (H^2)}{d(d+2)} G(y) +\frac{2\pi }{d m^2}\left( \sum_{1\leq i,k \leq m}\|{t_i-t_{\sigma_k(i)}}\|^2 \right) F(y)\right] 
\end{equation}
with 
\begin{equation}
F(y)=\sum_{\alpha >0}\left(\pi y\alpha-\dfrac{d}{2}\right) \abs{\Lambda_0(\alpha)} e^{-\pi y\alpha} 
\end{equation}
and 
\begin{equation}
G(y)=\sum_{\alpha >0} \pi \alpha\left(\pi y \alpha-\left( \frac{d}{2}+1\right) \right) \abs{\Lambda_0(\alpha)} e^{-\pi y \alpha}
\end{equation}
so that, in particular, 
\begin{equation}\label{deriv} 
G(y)=-\frac{d}{dy}F(y). 
\end{equation}
With these notations, the assertion that the lattice $\Lambda_0$ is  locally universally optimal \emph{among lattices} means that $G(y) > 0$ for all $y>0$. If this is the case, equation (\ref{deriv}) implies that $F$ is strictly decreasing on $\left] 0, + \infty\right)$; But then $F(y)$ is positive for any $y>0$ since, as we already observed, $F(y)$ is positive for any big enough $y$, \eg for $y >\dfrac{d}{2\pi\min \Lambda_0}$.
In other words, we have proved
\begin{corollary}
A lattice with all of its shells being $4$-designs is locally universally optimal \emph{among all periodic sets}, if and only if it is locally universally optimal \emph{among lattices}.
\end{corollary}
This principle applies in particular to $\mathsf{D}_4$, $\mathsf{E}_8$ and the Leech lattice, for which the $4$-design conditions are well-known to hold (there are basically two proofs, one using the properties of the automorphism group, and the other one using theta series with spherical coefficients, see \cite{MR2272103} for details). Altogether, we obtain
\begin{theorem}\label{thm:application}
The root lattices $\mathsf{D}_4$, $\mathsf{E}_8$ and the Leech lattice are locally universally optimal, \ie they locally minimize the $f$-energy on $\mathcal P ^{d,m}_{>0}$ for any $m$ and any completely monotonic potential $f$. 
\end{theorem}
\begin{proof}
 Applying the previous corollary, it is enough to check that these three lattices are locally universally optimal among lattices. But this is precisely Sarnak's and Str\"ombergsson's Proposition~2 in \cite{MR2221138}. 
The computations on pages 138--139 of their paper show that the $H$-part of the Hessian, $$H \mapsto \frac{\Tr (H^2)}{d(d+2)}\sum_{0\neq w \in \Lambda_0} c{\|w\|^2}\left(c{\|w\|^2}-\left( \frac{d}{2}+1\right) \right) e^{-c{\|w\|^2}}$$ is positive definite for any $c>0$.
\end{proof}

\section*{Acknowledgements}

We thank Giovanni Lazzarini for pointing out a mistake in a preliminary version of this paper. We also wish to thank the anonymous referee for various remarks and corrections which contributed to improve this work. This research collaboration was supported by the Universit\'e Bordeaux 1
and the NWO bezoekersbeurs 040.11.170.

%
%
%


\providecommand{\bysame}{\leavevmode\hbox to3em{\hrulefill}\thinspace}
\providecommand{\MR}{\relax\ifhmode\unskip\space\fi MR }
\providecommand{\MRhref}[2]{%
  \href{http://www.ams.org/mathscinet-getitem?mr=#1}{#2}
}
\providecommand{\href}[2]{#2}

\end{document}